\documentclass[reqno]{amsart}

\newtheorem{thm}{Theorem}

\newtheorem{cor}[thm]{Corollary}
\newtheorem{prop}[thm]{Proposition}
\theoremstyle{definition}

\newtheorem{ex}[thm]{Example}
\newtheorem{rem}[thm]{Remark}
\newtheorem{lem-app}{Lemma}[section]
\newtheorem{thm-app}[lem-app]{Theorem}

\begin{document}


\title[Hypergeometric RRU]{Central Limit Theorems for a\\ Hypergeometric
  Randomly Reinforced Urn} 

\author{Irene Crimaldi} 

\address{Irene
  Crimaldi, IMT Institute for Advanced Studies, Piazza San Ponziano 6,
  55100 Lucca, Italy} 

\email{irene.crimaldi@imtlucca.it}

\keywords{Central Limit Theorem; P\'olya urn; Randomly Reinforced Urn;
  Stable Convergence}

\subjclass[2010]{60F05; 60G57; 60B10; 60G42}

\date{\today}

\begin{abstract} 
We consider a variant of the randomly reinforced urn where more balls
can be simultaneously drawn out and balls of different colors can be
simultaneously added. More precisely, at each time-step, the
conditional di\-stri\-bu\-tion of the number of extracted balls of a
certain color given the past is assumed to be hypergeometric. We prove
some central limit theorems in the sense of stable convergence and of
almost sure conditional convergence, which are stronger than
convergence in distribution. The proven results provide asymptotic
confidence intervals for the limit proportion, whose distribution is
generally unknown. Moreover, we also consider the case of more urns
subjected to some random common factors.
\end{abstract}

\maketitle

\section{Introduction}
\label{intro}

Urn models, also known as preferential attachment models, are
stochastic processes in which, along the time-steps, different
individuals or objects or categories (represented by different colors)
receive some quantity, called ``weight'' (represented by the number of
balls), in such a way that the higher the total weight they already
have until a certain time, the greater the probability of receiving an
additional weight at the next time (i.e. a ``self-reinforcing''
property).  The preferential attachment is a key feature governing the
dynamics of many biological, economic and social systems. Therefore,
urn models are a very popular topic because of their hints for
theoretical research and their applications in various fields:
clinical trials (e.g. \cite{BHR, CL, DFL, HR, MF, z3}), economics and
finance (e.g. \cite{beggs, ER, HP}), information science
(e.g. \cite{mah-2008, MH}), network theory (e.g. \cite{BCM, CCCP,
  colle}) and so on.

\indent The first example of urn scheme is the standard
Eggenberger-P\'olya urn \cite{P23, P31}: an urn contains $a$ red and
$b$ black balls and, at each discrete time, a ball is drawn out from
the urn and then it is put again inside the urn together with an
additional constant number $k>0$ of other balls of the same color. Let
$Z_n$ be the proportion of red balls at time $n$, namely, the
conditional probability of drawing a red ball at time $n+1$, given the
outcomes of the previous extractions. A well known result (see, for
instance, \cite{mah-2008}) states that $(Z_n)$ is a bounded martingale
and $Z_n$ converges almost surely to a random variable $Z$ with Beta
distribution with parameters $a/k$ and $b/k$.

\indent Subsequently, urn models have been widely studied by many
researchers and there is a rather extensive literature on them
(e.g. \cite{AMS, BH, BCPR-urne1, BCPR-urne-dom, Bose, CPS, Chen, Das,
  J05, LP, MPS, z1}): a large number of new ``replacement policies''
(for instance, balanced rules, tenable mechanisms and random
reinforcements) and various related models (for instance, the
Poisson-Dirichlet model \cite{BCL} and the very recent Indian buffet
model \cite{BCPR-indian}) have been introduced and analyzed from
different points of view and by means of different techniques
(combinatorial methods, martingales, branching processes, stochastic
approximations, etc.).  We refer to \cite{P}, and the references
therein, for a general survey on random processes with reinforcement.

\indent In particular, as an extension of the P\'olya urn, the
Randomly Reinforced Urn (RRU) was recently proposed and analyzed
\cite{AMS, BCPR-urne1, BCPR-urne-dom, BCPR, BPR, C, CL, MF, MPS, z2,
  z1}. It consists in a multicolor urn which is reinforced at each
time with a random number of additional balls according to the color
of the extracted ball. The distribution of the reinforcements may
depend on time and be different for the different colors. These models
are suitable in order to describe the evolution of some system, such
as a population, and also to perform an adaptive design, i.e. an
experimental design that uses accumulated data to decide on how to
carry on the study, without undermining the validity and the integrity
of the experiment. Indeed, the RRU model provides randomized treatment
allocation schemes (clinical trials) where patients are assigned to
the best treatment with probability converging to one
\cite{BCPR-urne-dom, MF}.

\indent In \cite{perron} a new version of the RRU model is
formulated. This model consists of an urn which contains balls of two
different colors, say $a\in{\mathbb N}\setminus\{0\}$ balls of color A
and $b\in{\mathbb N}\setminus\{0\}$ balls of color B. At each time
$n\geq 1$, we simultaneously (i.e. without replacement) draw a random
number $N_n$ of balls. Let $X_n$ be the number of extracted balls of
color A.  Then we return the extracted balls in the urn together with
other $R_nX_n$ balls of color A and $R_n(N_n-X_n)$ balls of color
B. The size $R_n$ of the reinforcement is assumed independent of
$[N_1,X_1,R_1,\ldots,N_{n-1}, X_{n-1},R_{n-1}, N_n, X_n]$. We will
call this model ``Hypergeometric Randomly Reinforced Urn'' (HRRU).

\indent With respect to the RRU model, the main novelties of this
model are that, at each time, more balls can be simultaneously drawn
out (and returned in the urn) and balls of different colors can be
simultaneously added. The number of extracted balls of a certain color
depends on the composition of the urn at the moment of the extraction,
akin a preferential attachment rule. When $N_n=1$ for each $n$, the
HRRU reduces to the RRU model with equal reinforcements for the two
colors. In particular, the case $N_n=1$ and $R_n=k$ (where $k$ is a
constant) for each $n$ corresponds to the standard Eggenberger-P\'olya
urn; while the case $N_n=h$ and $R_n=k$ (where $h$ and $k$ are two
constants) for each $n$ coincides with the model in \cite{Chen,
  Chenb}. Also the model introduced and studied in \cite{CL} can be
seen as a RRU model where balls of different colors can be
simultaneously added, but there the ``multi-updating'' is due to a
delay in the updating. Indeed, at each time $n$ a single ball is drawn
out (and returned in the urn) but the updating is performed at certain
time-steps $(u_i)_{i\geq 1}$ as follows: at time $u_i$, we add a
random number $R_n$ of balls of the same color of the ball extracted
at time $n$, for each $n=u_{i-1}+1,\dots, r_i$, with $r_i\leq u_i$.

\indent As explained in \cite{perron}, a possible interpretation of
the HRRU model is the following. At each time $n\geq 1$, a new firm
appears on the market and it has to choose the operative system for
its computers among two different types, say operative system A (to
which we associate color A) and operative system B (to which we
associate color B). The total number of its computers is $R_nN_n$
(more precisely, $N_n$ blocks of $R_n$ computers each). The firm
decides to adopt $X_n$ blocks (of size $R_n$ each) with operative
system A and $(N_n-X_n)$ blocks (of size $R_n$ each) with operative
systems B, according to the number of computers with operative systems
$A$ already present in the market. Another possible interpretation
follows.  At each time $n\geq 1$, a pharmaceutical firm has to select
the size of its production for two different kinds of products, say
product A and product B. For instance, A and B can be two medicines
for the same disease but with different costs. The total of its
production is $R_nN_n$ (more precisely, the firm produces $N_n$
blocks, each of size $R_n$). The firm decides to produce $X_n$ blocks
of type A-products and $(N_n-X_n)$ blocks of type B-products according
to the number of type A-products already on the market. Finally,
setting $R_n=1$ for each $n$, the HRRU model can be employed to
describe the growth of a population in which we can distinguish two
types of individuals, say A and B. At each time $n$, the random
numbers $N_n$ and $X_n$ represent the new offsprings and the new
offsprings of type A, respectively. The number of the new type
A-individuals depends on the composition of the population at the
preceeding time-step.

\indent It is shown in \cite{perron}, that $Z_n$ converges almost
surely to a random variable $Z$, whose distribution is generally
unknown. Authors also provide some results concerning the distribution
of the limit random variable $Z$ in some particular cases.  In the
present paper we continue the study of the model proving some central
limit theorems and making another step toward the description of the
distribution of $Z$. Further, the proven central limit theorems can be
used in order to obtain asymptotic confidence intervals for the limit
proportion $Z$. Moreover, we can also consider the case of more urns
(for instance, according to the previous interpretations, the
different urns can re\-pre\-sent different markets or different
populations), each of them following a HRRU dynamics, and perform some
test for comparing them or get asymptotic confidence intervals for any
linear combination of the limit proportions.

\indent The paper is organized as follows. In Section \ref{model} we
formally introduce the model. In Section \ref{stable} we recall the
needed facts concerning stable convergence and almost sure conditional
convergence. In Section \ref{results} we give and discuss the main
results, whose proofs are postponed to Section \ref{proofs}. Finally,
in Section \ref{stat} we provide some statistical tools based on the
proven results. The paper is enriched with an appendix which contains
some useful auxiliary results.

\section{The HRRU model}
\label{model}

An urn contains $a\in{\mathbb N}\setminus\{0\}$ balls of
color A and $b\in{\mathbb N}\setminus\{0\}$ balls of color B. At each
time $n\geq 1$, we simultaneously (i.e. without replacement) draw a
random number $N_n$ of balls. Let $X_n$ be the number of extracted
balls of color A.  Then we return the extracted balls in the urn
together with other $R_nX_n$ balls of color A and $R_n(N_n-X_n)$ balls
of color B. More precisely, we take a 
probability space 
$(\Omega,\mathcal{A},P)$
and, on it, some random variables  
$N_n,\,X_n,\,R_n$ 
 such that, for each $n\geq 1$, we have:

\begin{itemize}
\item[i)] The conditional distribution of the random variable $N_n$ given 
$$[N_1,X_1, R_1,\dots,N_{n-1}, X_{n-1}, R_{n-1}]$$ is concentrated on
  $\{1,\dots,S_{n-1}\}$ where
\begin{equation}
S_{n-1}=a+b+\sum_{j=1}^{n-1} N_j R_j
=\hbox{total number of balls at time } n-1.
\end{equation}

\item[ii)] The conditional distribution of the random variable $X_n$ given 
$$[N_1,X_1, R_1,\dots,N_{n-1}, X_{n-1}, R_{n-1}, N_n]$$ is
  hypergeometric with parameters $N_n,\, S_{n-1}$ and
  $H_{n-1}$\footnote{ We recall that a random variable $X$ has
    hypergeometric distribution with parameter $N,S, H$ if
    $P\{X=k\}=\frac{\binom{H}{k}\binom{S-H}{N-k}}{\binom{S}{N}}$}
  where
\begin{equation}
H_{n-1}=a+\sum_{j=1}^{n-1} X_j R_j
=\hbox{total number of balls of color A at time } n-1.
\end{equation}

\item[iii)] The random variable $R_{n}$ takes values in ${\mathbb
  N}\setminus\{0\}$ and it is independent of
 $$[N_1,X_1,R_1,\ldots,N_{n-1}, X_{n-1},R_{n-1}, N_n, X_n].$$
\end{itemize}

\indent Note that we do not specify the conditional
distribution of $N_n$ given the past $[N_1,X_1, R_1,\dots,N_{n-1}, X_{n-1},
  R_{n-1}]$ nor the distribution of $R_n$.\\

\indent We will refer to the above urn model as the {\em
  Hypergeometric Randomly Reinforced Urn} (HRRU)\footnote{ It
  coincides with the model introduced in \cite{perron} but here the
  adopted notation is different: $M_n$ in \cite{perron} corresponds to
  our $N_n$ (total number of extracted balls at time $n$), $R_n$ in
  \cite{perron} corresponds to our $X_n$ (number of extracted balls of
  color A at time $n$) and $N_n$ in \cite{perron} corresponds to our
  $R_n$ (number of added balls for each extracted ball at time
  $n$). We decided to adopt a different notation with respect to
  \cite{perron} in order to use a notation more similar to the one
  used in the RRU model literature.}. It is worthwhile to remark that
this model include the classical P\'olya urn (the case with $N_n=1$
and $R_n=k$ for each $n$) and the randomly reinforced urn with the
same reinforcements for both colors (the case with $N_n=1$ for each
$n$ and $R_n$ arbitrarily random).\\

\indent We set $Z_n$ equal to the proportion of balls of color A in the
urn (immediately after the updating of the urn at time $n$ and
immediately before the $(n+1)$-th extraction), that is $Z_0=a/(a+b)$
and
\begin{gather*}
  Z_n=\frac{H_n}{S_n}\quad\hbox{for } n\geq 1.
\end{gather*}
Moreover we set 
\begin{equation*}
\mathcal{F}_0=\{\emptyset, \Omega \},\quad\mathcal{F}_n=
\sigma\bigl(N_1,X_1,R_1,\ldots,N_n, X_n,R_n\bigr)\quad\hbox{for } n\geq 1\,,
\end{equation*}
and 
\begin{equation*}
\mathcal{G}_n=\mathcal{F}_n\vee\sigma(N_{n+1}),\quad
\mathcal{H}_n=\mathcal{G}_n\vee\sigma(R_{n+1})\quad\hbox{for } n\geq 0.
\end{equation*}

\section{Stable convergence and almost sure conditional convergence}
\label{stable}

Stable convergence has been introduced by R\'enyi in \cite{renyi-1963}
and subsequently in\-ve\-sti\-ga\-ted by various authors, e.g.
\cite{AE, cri-let-pra-2007, feigin-1985, jacod-1981, peccati-2008}. It
is a strong form of convergence in distribution, in the sense that it
is intermediate between the simple convergence in distribution and the
convergence in probability. In this section we recall some basic
definitions and properties. For more details, we refer the reader to
\cite{cri-let-pra-2007, hall-1980} and the references therein.\\

\indent Let $(\Omega, {\mathcal A}, P)$ be a probability space and let
$S$ be a Polish space (i.e. a completely metrizable separable
topological space), endowed with its Borel $\sigma$-field. A {\em
  kernel} on $S$, or a random probability measure on $S$, is a
collection $K=\{K(\omega,\cdot):\, \omega\in\Omega\}$ of probability
measures on the Borel $\sigma$-field of $S$ such that, for each
bounded Borel real function $f$ on $S$, the map
$$
\omega\mapsto 
K\!f(\omega)=\int f(x)\, K(\omega, dx) 
$$ is $\mathcal A$-measurable. Given a kernel on $S$ and an event $H$
in $\mathcal A$ with $P(H)>0$, we can define a probability measure on
$S$, denoted by $P_HK$, as follows:
$$
P_HK(B)=E[K(\cdot,B)|H]=P(H)^{-1}\int_H K(\omega, B)\, P(d\omega),
$$
for each Borel set $B$ of $S$. We simply write $PK$ when $H=\Omega$. 
It is easy to verify the relation
$$
\int f(x)\,P_HK(dx)=P(H)^{-1}\int_H K\!f(\omega)\,P(d\omega).
$$

\indent On $(\Omega, {\mathcal A},P)$ let $(Y_n)$ be a sequence of
$S$-valued random variables and let $K$ be a kernel on $S$. Then we
say that $Y_n$ converges {\em stably} to $K$, and we write
$Y_n\stackrel{stably}\longrightarrow K$, if
$$
P(Y_n \in \cdot \,|\, H)\stackrel{weakly}\longrightarrow 
P_HK
\qquad\hbox{for all } H\in{\mathcal A}\; \hbox{with } P(H) > 0.
$$

\indent Clearly, if $Y_n\stackrel{stably}\longrightarrow K$, then
$Y_n$ converges in distribution to the probability measure $PK$.
Moreover, we recall that the convergence in probability of $Y_n$ to a
random variable $Y$ is equivalent to the stable convergence of $Y_n$
to a special kernel, which is the Dirac kernel $K=\delta_Y$.\\

\indent We next mention a form of convergence, called almost sure
conditional convergence, introduced and studied in \cite{C}, and
afterwards employed by other researchers (see, for example, \cite{AMS,
  z2}).\\

\indent For each $n$, let ${\mathcal F}_n$ be a sub-$\sigma$-field of
$\mathcal A$ and set ${\mathcal F}=({\mathcal F}_n)$ (called
conditioning system). If $K_n$ denotes a version of the conditional
distribution of $Y_n$ given ${\mathcal F}_n$, we say that $Y_n$
converges to $K$ in the sense of the {\em almost sure conditional
  convergence} with respect to ${\mathcal F}$, if, for almost every
$\omega$ in $\Omega$, the probability measure $K_n(\omega, \cdot)$
converges weakly to $K(\omega,\cdot)$. Evidently, if
$Y_n$ converges to $K$ in the sense of the almost sure conditional
convergence with respect to ${\mathcal F}$, we have that 
$$
E\left[f(Y_n)\,|\,{\mathcal F}_n\right]\stackrel{a.s.}\longrightarrow K\!f
$$ for each bounded continuous real function $f$ on $S$ and $Y_n$
converges in distribution to the probability measure $PK$.\\

\indent In the sequel we will adopt the notation ${\mathcal N}(0,V)$
in order to indicate the {\em Gaussian kernel} with zero mean and
random variance $V$, that is the collection $\{ {\mathcal N}(0,
V(\omega) ):\, \omega\in \Omega\}$ of centered Gaussian distributions,
where $V$ is a positive random variable (${\mathcal N}(0,0)$ is meant
as the Dirac probability measure concentrated in zero). Further, given
two kernels $K_1$ and $K_2$, we will denote by $K_1\otimes K_2$ the
kernel given by the product measures $K_1(\omega,\cdot)\otimes
K_2(\omega, \cdot)$.

\section{Convergence results for the HRRU model}
\label{results}

The sequence $(Z_n)$ is a bounded $\mathcal H$-martingale. Indeed, we have 
\begin{equation}\label{eq-base} 
Z_{n}-Z_{n-1}=\frac{R_{n} (X_{n}-N_{n}Z_{n-1})}{S_{n}}
\end{equation} 
and so
\begin{equation*}
\begin{split}
E[Z_{n}-Z_{n-1} |{\mathcal H}_{n-1}]
&=
\frac{R_{n}}{S_{n}}\left(E[X_{n} |{\mathcal H}_{n-1}]-N_{n}Z_{n-1}\right)
=
\frac{R_{n}}{S_{n}}\left(E[X_{n} |{\mathcal G}_{n-1}]-N_{n}Z_{n-1}\right)\\
&=0
\end{split}
\end{equation*}
(where the second equality holds true because of condition iii) and
the last one is implied by condition ii)).  Hence, the sequence
$(Z_n)$ converges almost surely (and in $L^1$) to a random variable
$Z$.  Lemma \ref{lemma-app} (with $Y_n=X_n/N_n$) immediately implies
that the sequence
\begin{equation}
M_n=\frac{1}{n}\sum_{j=1}^{n}\frac{X_j}{N_j}
\end{equation}
also converges almost surely (and in $L^1$) to $Z$ (cfr. Th.~3.1,
Th.~3.5 in \cite{perron}). \\

\indent The distribution of $Z$ is unknown except in a few particular
cases (see \cite{perron}). We are going to prove the following central
limit theorems, useful in order to get some information on $Z$.

\begin{thm}\label{main-1}
Assume there exists a constant $k\in{\mathbb N} \setminus\{0\}$ such
that $N_n\vee R_n\leq k$ for each $n$ and
\begin{equation}\label{ass-base}
E[N_n|{\mathcal F}_{n-1}]\stackrel{a.s.}\longrightarrow N,\quad 
E[R_n]\longrightarrow m,\quad  E[R_n^2]\longrightarrow q\,,
\end{equation}
where $N$ is a strictly positive bounded random variable and $m$ and
$q$ are finite and strictly positive numbers.\\ \indent Then
$\sqrt{n}(Z_n-Z)$ converges in the sense of the almost sure
conditional convergence with respect to ${\mathcal F}=({\mathcal F}_n)$
to the Gaussian kernel ${\mathcal N}(0, V)$, where
$$V=q m^{-2} N^{-1} Z(1-Z).$$
\end{thm}

\begin{thm}\label{main-2}
Under the assumptions of Theorem \ref{main-1}, suppose also that  
\begin{equation}\label{ass-base-bis}
E[N_n^{-1}|{\mathcal F}_{n-1}]\stackrel{a.s.}\longrightarrow \eta,
\end{equation}
where $\eta$ is a strictly positive bounded random variable.\\ 
\indent Then 
$$
[ \sqrt{n}(M_n-Z_n), \sqrt{n}(Z_n-Z) ] 
\stackrel{stably}\longrightarrow 
{\mathcal N}(0, U)\otimes {\mathcal N}(0, V),
$$ 
where
$$
 U=V+Z(1-Z)\left(\eta-2N^{-1}\right)=
\left(q m^{-2} N^{-1}+\eta-2N^{-1}\right)Z(1-Z).
$$ 
\end{thm}

From the above theorems we have that $\sqrt{n}(M_n-Z_n)$ converges
stably to ${\mathcal N}(0, U)$ and $\sqrt{n}(M_n-Z)$ converges stably
to ${\mathcal N}(0, U+V)$.\\

\indent The following corollary enriches Corollary 3.4 in
\cite{perron}.

\begin{cor}\label{cor-legge} 
Assume there exists a constant $k\in{\mathbb N} \setminus\{0\}$ such that
$N_n\vee R_n\leq k$ for each $n$. Then:
\begin{itemize}
\item[a)] $P(Z=0)+P(Z=1)<1$.
\item[b)] If assumptions (\ref{ass-base}) are also satisfied, then
  $P(Z=z)=0$ for all $z\in (0,1)$.
\end{itemize}
\end{cor}

Note that the above result entails that the limit Gaussian kernel in
Theorem \ref{main-1} is not degenerate.

Some examples and comments follow.

\begin{ex}
If $N_n=h_n$ with $h_n\in{\mathbb N}\setminus\{0\}$ and $h_n\uparrow
h\leq a+b$, then the first condition in (\ref{ass-base}) and condition
(\ref{ass-base-bis}) are obviously satisfied with $N=h$ and
$\eta=h^{-1}$, so that we have $V=q m^{-2} h^{-1} Z(1-Z)$ and 
$U=(q m^{-2}-1)h^{-1}Z(1-Z)$.
\end{ex}

\begin{rem}\label{remark-example}
\rm If $(N_n)$ is a sequence of integer-valued random variables with
$1\leq N_n\leq k$ and converging almost surely to a random variable
$N$, then (by Lemma \ref{lemma-app-C}) the first condition in
(\ref{ass-base}) holds true. Moreover, condition (\ref{ass-base-bis})
is satisfied with $\eta=N^{-1}$ and so we have
$U=(qm^{-2}-1)N^{-1}Z(1-Z)$.
\end{rem}

The next example concerns the above remark.

\begin{ex}
Suppose that $(N_n)$ is given by a symmetric random walk with two
absorbing barriers. More precisely, given $h\in{\mathbb N}$, with
$2\leq h \leq a+b$, set
$$
\widetilde{N}_1=i \in\{2,\dots, h-1\}, \qquad
\widetilde{N}_n=i+\sum_{j=1}^{n-1} Y_j
$$ 
\noindent where each $Y_j$ is independent of $[X_1,R_1,Y_1, X_2,
  R_2,\dots, Y_{j-1}, X_j, R_j]$ and such that
$P(Y_j=-1)=P(Y_j=1)=1/2$.  Set $\Gamma_1=0$ and
$\Gamma_n=\sum_{j=1}^{n-1} Y_j$ for $n\geq 2$, and define
\begin{equation*}
\begin{split}
T_1&=\inf\{n:\widetilde{N}_n=1\}=\inf\{n: \Gamma_n=1-i\}\\
T_h&=\inf\{n:\widetilde{N}_n=h\}=\inf\{n: \Gamma_n=h-i\}.
\end{split}
\end{equation*}
Finally, for each $n\geq 1$, set $N_n=\widetilde{N}_{T\wedge n}$
where $T=T_1\wedge T_h$. Then $N_n\stackrel{a.s.}\longrightarrow
N=\widetilde{N}_T$ where $N=I_{\{T=T_1\}}+hI_{\{T=T_h\}}$. In order to
find the probabilities $P(T=T_1)=p$ and $P(T=T_h)=1-p$, it is enough
to observe that, since $(\Gamma_n)$ is a martingale, we have
$$
E[\Gamma_T]=(1-i)p+(h-i)(1-p)=0
$$
and so $p=(h-i)/(h-1)$. According to Remark \ref{remark-example},
$\eta=N^{-1}=I_{\{T=T_1\}}+h^{-1}I_{\{T=T_h\}}$. 
\end{ex}

The last example regards the case when the random variables $N_n$ are
independent and identically distributed.

\begin{ex}
Suppose that $(N_n)$ are a sequence of random variables such that each
$N_n$ is independent of ${\mathcal F}_{n-1}$ and uniformly distributed
on the set $\{1,\dots, h\}$, with $2\leq h\leq a+b$. Then
$N=E[N_n]=(h+1)/2$ and $\eta=E[N_n^{-1}]=h^{-1}\sum_{j=1}^h j^{-1}$.
\end{ex}

\section{Proofs}
\label{proofs}

We begin with a preliminary result.

\begin{prop}
Assume there exists a constant $k\in{\mathbb N}\setminus\{0\}$ such
that $N_n\vee R_n\leq k$ for each $n$ and
\begin{equation}\label{ass-base-2}
E[N_n|{\mathcal F}_{n-1}]\stackrel{a.s.}\longrightarrow N,
\quad
E[R_n]\longrightarrow m,
\end{equation}
where $N$ is a strictly positive bounded random variable and $m$ is a
finite and strictly positive number.\\ \indent Then
$$
\frac{S_n}{n}\stackrel{a.s.}\longrightarrow N m.
$$
\end{prop}

\begin{proof} It follows from  Lemma \ref{lemma-app} with $Y_j=N_jR_j$. Indeed,
we have $Y_j^2\leq k^4$ for each $j$ and (by iii))
$$
E[N_j R_j |{\mathcal F}_{j-1}]=E[N_j|{\mathcal F}_{j-1}] E[R_j]
\stackrel{a.s.}\longrightarrow N m\,.
$$
\end{proof}

\noindent{\it Proof of Theorem \ref{main-1}.} Setting $X'_n=X_n/N_n$
for each $n$, the sequence $(X'_n)$ is $\mathcal G$-adapted and
bounded. Moreover, we have
\begin{equation}\label{eq-p0}
\begin{split}
E[X_{n+1}'|{\mathcal G}_n]=
E[ N_{n+1}^{-1} X_{n+1}|{\mathcal G}_n]=  N_{n+1}^{-1} E[X_{n+1}|{\mathcal G}_n]=
N_{n+1}^{-1} N_{n+1} Z_n=Z_n 
\end{split}
\end{equation}
and, as we have already said, the sequence $(Z_n)$ is a bounded
$\mathcal G$-martingale. Therefore, in order to prove Theorem
\ref{main-1}, it suffices to prove that the following conditions are
satisfied (see Theorem \ref{th-main-app} applied to $Y_n=X_n'$):
\begin{itemize}
\item[c1)] $E[\sup_{j\geq 1} \sqrt{j} |Z_{j-1}-Z_j|\,]<+\infty$;
\item[c2)] $n\sum_{j\geq n}(Z_{j-1}-Z_j)^2\stackrel{a.s.}\longrightarrow V$ 
for some random variable $V$.
\end{itemize}
In the following we verify the above conditions.\\

\noindent{\em Condition c1).} We observe that equality (\ref{eq-base})
can be rewritten as
\begin{equation}\label{eq-p1} 
Z_{j-1}-Z_j =
\frac{R_j N_j (Z_{j-1}-X'_j)}{S_j}\,,
\end{equation}
so that we find 
\begin{equation}\label{eq-p2}
|Z_{j-1}-Z_j|\leq \frac{k^2}{j}\,.
\end{equation}
Therefore condition c1) is obviously verified.
\\

\noindent{\em Condition c2).} We want to apply Lemma \ref{lemma-app}
with $Y_j=j^2(Z_{j-1}-Z_j)^2$.  By the assumptions and inequality
(\ref{eq-p2}), we have $\sum_j 
j^{-2}E[Y_j^2]<+\infty$. Moreover, by equality (\ref{eq-p1}), we have
\begin{equation*}
E[Y_{j}|{\mathcal F}_{j-1}] = j^2 E[(Z_{j-1}-Z_j)^2|{\mathcal F}_{j-1}]
=j^2 E[S_j^{-2} R_j^2 N_j^2 (Z_{j-1}-X'_j)^2|{\mathcal F}_{j-1}],
\end{equation*}
and so (by iii)) we get the two inequalities 
\begin{equation*}
\begin{split}
& E[Y_{j}|{\mathcal F}_{j-1}] \geq 
\frac{j^2}{(S_{j-1}+ k^2)^2} 
E[R_j^2] E[N_j^2 (Z_{j-1}-X'_j)^2|{\mathcal F}_{j-1}]\\
& 
E[Y_{j}|{\mathcal F}_{j-1}] \leq 
\frac{j^2}{S_{j-1}^2} E[R_j^2] E[N_j^2 (Z_{j-1}-X'_j)^2|{\mathcal F}_{j-1}].
\end{split}
\end{equation*}
Since $S_n/n\stackrel{a.s.}\longrightarrow N m$ and $E[R_j^2]$ converges to
$q$, it is enough to prove the almost sure convergence of $E[N_j^2
  (Z_{j-1}-X'_j)^2|{\mathcal F}_{j-1}]$ to $N Z(1-Z)$. To this purpose,
we observe that we can write 
\begin{equation*}
E[N_j^2 (Z_{j-1}-X'_j)^2|{\mathcal F}_{j-1}]=
E\left[ N_j^2 
E[(Z_{j-1}-X'_j)^2 |{\mathcal G}_{j-1}]
\,|\,{\mathcal F}_{j-1}\right]
\end{equation*}
and, by ii) and relation (\ref{eq-p0}), the conditional expectation
$E[(Z_{j-1}-X'_j)^2 |{\mathcal G}_{j-1}]$ coin\-ci\-des with
\begin{equation*}
\begin{split}
&Z_{j-1}^2+N_j^{-2} E[X_j^2|{\mathcal G}_{j-1}]
-2 Z_{j-1} E[X'_j|{\mathcal G}_{j-1}]
=\\
&Z_{j-1}^2+ N_j^{-2}
\left[
Z_{j-1}(1-Z_{j-1})(S_{j-1}-1)^{-1} N_{j}(S_{j-1}-N_j)
+Z_{j-1}^2 N_j^2\right]
-2 Z_{j-1}^2
=\\
&
Z_{j-1}(1-Z_{j-1})(S_{j-1}-1)^{-1}N_j^{-1}\left(S_{j-1}-N_j\right).
\end{split}
\end{equation*}
Therefore we obtain 
\begin{equation*}
E[N_j^2 (Z_{j-1}-X'_j)^2|{\mathcal F}_{j-1}]
=Z_{j-1}(1-Z_{j-1})(S_{j-1}-1)^{-1}
\left(S_{j-1} E[N_j|{\mathcal F}_{j-1}]-E[N_j^2|{\mathcal F}_{j-1}]\right),
\end{equation*}
which converges to $N Z(1-Z)$ (since $E[N_j^2|{\mathcal F}_{j-1}]$ is
bounded by $k^2$ and $S_{j-1}\stackrel{a.s.}\longrightarrow +\infty$).
Hence $E[Y_j|{\mathcal F}_{j-1}]$ converges almost surely to $V$ and,
by Lemma \ref{lemma-app}, condition c2) is satisfied.\\
\indent The proof is so concluded.

{\it Proof of Theorem \ref{main-2}.} Thanks to what we have already
proven in the previous proof, it suffices to verify that the following
condition is satisfied (see Theorem \ref{th-main-app}
applied to $Y_n=X_n'$):
\begin{itemize}
\item[c3)] $n^{-1} \sum_{j=1}^n
  \big[X'_j-Z_{j-1}+j(Z_{j-1}-Z_j)\big]^2
  \stackrel{P}\longrightarrow U$ for some random variable $U$.
\end{itemize}

To this purpose, we apply Lemma \ref{lemma-app} with
$$Y_j=\big[ X'_j - Z_{j-1}+ j(Z_{j-1}-Z_j) \big]^2.$$ Indeed, by the assumptions
and inequality (\ref{eq-p2}), we have $\sum_j
j^{-2}E[Y_j^2]<+\infty$. Moreover, from what we have already seen in
the previous proof, we can get
\begin{equation*}
j^2E[(Z_{j-1} - Z_{j})^2|{\mathcal F}_{j-1}]
\stackrel{a.s.}\longrightarrow V\,,
\end{equation*}
\begin{equation*}
E[(X'_j-Z_{j-1})^2|{\mathcal F}_{j-1}]
\stackrel{a.s.}\longrightarrow
\eta Z(1-Z)
\end{equation*}
and, with a similar arguments,  
\begin{equation*}
\begin{split}
2jE[(X_j'-Z_{j-1})(Z_{j-1}-Z_j)|{\mathcal F}_{j-1}]&=
-2jE\left[S_j^{-1} R_jN_j (Z_{j-1}-X_j')^2|{\mathcal F}_{j-1}\right]
\\
&\stackrel{a.s.}\longrightarrow
-2 N^{-1} Z(1-Z).
\end{split}
\end{equation*}

\noindent{\it Proof of Corollary \ref{cor-legge}.}  
Assertion a) is proven in Corollary 3.4. in \cite{perron}. Let us
prove assertion b) arguing as in \cite{z2}.\\

\indent Let $A$ be a $\bigvee_{n}{\mathcal F}_n$-measurable event and
set $I_n=E[I_A|{\mathcal F}_n]$. Then
$I_n\stackrel{a.s.}\longrightarrow I_A$. By Lemma \ref{lemma-app-C}, we find 
\begin{equation}\label{eq1-cor}
E\big[(I_A-I_n)\exp(it\sqrt{n}(Z_n-Z))|{\mathcal F}_n\big]
\stackrel{a.s.}\longrightarrow 0.
\end{equation}
On the other hand,  by Theorem \ref{main-1}, we have
\begin{equation}\label{eq2-cor}
E\big[I_n\exp(it\sqrt{n}(Z_n-Z))|{\mathcal F}_n\big]=
I_n E\big[\exp(it\sqrt{n}(Z_n-Z))|{\mathcal F}_n\big]
\stackrel{a.s.}\longrightarrow I_A \exp(-(t^2 V)/2).
\end{equation}
Hence, from (\ref{eq1-cor}) and (\ref{eq2-cor}), we get
\begin{equation}\label{eq-c1}
E\big[I_A\exp(it\sqrt{n}(Z_n-Z))|{\mathcal F}_n\big]
\stackrel{a.s.}\longrightarrow \exp(-(t^2 V)/2)I_A.
\end{equation}
In order to conclude, it is enough to fix $z\in (0,1)$, take
$A=\{Z=z\}$ and observe that (\ref{eq-c1}) implies almost surely  
\begin{equation*}
\begin{split}
I_A\exp(-(t^2 V)/2)&=
\lim_n E\big[I_A\exp(it\sqrt{n}(Z_n-Z))|{\mathcal F}_n\big]\\
&=
\lim_nE\big[I_A\exp(it\sqrt{n}(Z_n-z))|{\mathcal F}_n\big]\\
&=
\lim_n I_n\exp(it\sqrt{n}(Z_n-z))=
\lim_n I_A\exp(it\sqrt{n}(Z_n-z))
\end{split}
\end{equation*}
and so almost surely 
$$ 
I_A=|\lim_n I_A\exp(it\sqrt{n}(Z_n-z))|=I_A\exp(-(t^2 V)/2).
$$ Since we have $V>0$ on $A$, it results $\exp(-(t^2 V)/2)<1$ on $A$
for $t\neq 0$ and so we necessarily conclude that $P(A)$ is zero.

\section{Statistical tools}
\label{stat}

\subsection{Asymptotic confidence intervals for the limit proportion} 
\label{conf-int}

By means of Theorem \ref{main-1} and Theorem \ref{main-2}, we can
construct asymptotic confidence intervals for the limit proportion
$Z$. More precisely, under the assumptions of Theorem \ref{main-1},
also assume $k\leq a+b$ (so that $N_{n}\leq S_{n-1}$ for each $n$). If
we are in the particular case when:
\begin{itemize}
\item for each $n$, the random variable $N_n$ is independent of
  ${\mathcal F}_{n-1}$ and all the random variables $N_n$ are
  identically di\-stri\-bu\-ted with mean value $\mu$ (so that
  $N=E[N_n]=\mu$ and $\eta=E[N_n^{-1}]$) and
\item all the random variables $R_n$ (that are independent by assumption
  iii)) are also identically distributed (so that $m=E[R_n]$ and
  $q=E[R_n^2]$),
\end{itemize}
then two asymptotic confidence intervals for $Z$ are
\begin{equation}
Z_n\pm q_{1-\frac{\alpha}{2}}\sqrt{\frac{V_n}{n}}
\qquad
M_n\pm q_{1-\frac{\alpha}{2}}\sqrt{\frac{W_n}{n}}
\end{equation}
where $q_{1-\frac{\alpha}{2}}$ is the quantile of order
$1-\frac{\alpha}{2}$ of the standard normal distribution and
\begin{equation}
V_n=\frac{q_n}{m_n^2\mu_n}Z_n(1-Z_n),\qquad
W_n=
\left(
\frac{2q_n}{m_n^2\mu_n}+\eta_n-\frac{2}{\mu_n}
\right)
M_n(1-M_n)
\end{equation}
with
\begin{equation}
\begin{split}
m_n&=\frac{\sum_{j=1}^n R_j}{n},\qquad
q_n=\frac{\sum_{j=1}^n R_j^2}{n}
\\
\mu_n&=\frac{\sum_{j=1}^n N_j}{n},\qquad
\eta_n=\frac{\sum_{j=1}^n N_j^{-1}}{n}.
\end{split}
\end{equation}

Note that the second interval does not depend on the initial
composition of the urn, which could be unknown.

\subsection{The case of more urns}
\label{two}

Let ${\mathcal U}$ be a finite set. Every index $u\in{\mathcal U}$
labels an urn initially containing $a(u)$ balls of color $A$ and
$b(u)$ balls of color $B$.  Each of the urn follows the dynamics
described in the Section \ref{model}.  For instance, according to the
interpretations given in Section \ref{intro}, we can see $\mathcal U$
as a set of different markets or different populations.\\

\indent More precisely, we take a probability space
$(\Omega,\mathcal{A},P)$ and, on it, some random vectors
$X_n=[X_n(u)]_{u\in \mathcal U}$, $N_n=[N_n(u)]_{u\in\mathcal U}$,
$R_n=[R_n(u)]_{u\in\mathcal U}$ such that, for each $n\geq 1$, we
have:

\begin{itemize}
\item[i)] The conditional distribution of the random vector $N_n$ given 
$$[N_1,X_1, R_1,\dots,N_{n-1}, X_{n-1}, R_{n-1}]$$ is concentrated on
  $\prod_{u\in {\mathcal U}}\{1,\dots,S_{n-1}(u)\}$ where
\begin{equation}
S_{n-1}(u)=a(u)+b(u)+\sum_{j=1}^{n-1} N_j(u) R_j(u).
\end{equation}

\item[ii)] The conditional distribution of the random vector $X_n$ given 
$$[N_1,X_1, R_1,\dots,N_{n-1}, X_{n-1}, R_{n-1}, N_n]$$ is 
the product 
$$\bigotimes_{u\in{\mathcal U}}
\hbox{Hypergeom}(N_n(u),S_{n-1}(u),H_{n-1}(u)),$$ where
$\hbox{Hypergeom}(N_n(u),S_{n-1}(u),H_{n-1}(u))$ denotes the hypergeometric
distribution with parameters $N_n(u),\, S_{n-1}(u)$ and
$H_{n-1}(u)$ with 
\begin{equation}
H_{n-1}(u)=a(u)+\sum_{j=1}^{n-1} X_j(u) R_j(u). 
\end{equation}

\item[iii)] The random vector $R_{n}$ takes values in 
$({\mathbb
  N}\setminus\{0\})^{card({\mathcal U})}$ and it is independent of
 $$[N_1,X_1,R_1,\ldots,N_{n-1}, X_{n-1},R_{n-1}, N_n, X_n].$$
\end{itemize}

\noindent We set $Z_n(u)$ equal to the proportion of balls of color A in
the urn $u$ (immediately after the updating of the urn at time $n$ and
immediately before the $(n+1)$-th extraction), that is
$Z_0(u)=a(u)/(a(u)+b(u))$ and
\begin{gather*}
  Z_n(u)=\frac{H_n(u)}{S_n(u)}\quad\hbox{for } n\geq 1.
\end{gather*}
Moreover we set 
\begin{equation*}
\mathcal{F}_0=\{\emptyset, \Omega \},\quad\mathcal{F}_n=
\sigma\bigl(N_1,X_1,R_1,\ldots,N_n, X_n,R_n\bigr)\quad\hbox{for } n\geq 1\,,
\end{equation*}
and 
\begin{equation*}
\mathcal{G}_n=\mathcal{F}_n\vee\sigma(N_{n+1})
\quad\hbox{for } n\geq 0.
\end{equation*}

From condition ii) follows that $X_n(u)$ and $X_n(v)$ are ${\mathcal
  G}_{n-1}$-conditionally independent for $u\neq v$ and so, setting
$X'_n(u)=X_n(u)/N_n(u)$ for each $n$ and $u$, we have 
\begin{equation}\label{cond-ind}
\begin{split}
&
E\left[\left(Z_{n-1}(u)-X'_n(u)\right)
\left(Z_{n-1}(v)-X'_n(v)\right)|{\mathcal G}_{n-1}\right]
=
\\
&
E[Z_{n-1}(u)-X'_n(u)|{\mathcal G}_{n-1}]
E[Z_{n-1}(v)-X'_n(v)|{\mathcal G}_{n-1}]
\\
&
=0.
\end{split}
\end{equation}

It is worthwhile to note that, for a given $n$, we are not assuming
the random variables $N_n(u)$ (resp. $R_n(u)$), with $u\in{\mathcal
  U}$, to be independent. For example, we can assume
$$
N_n(u)=h(u)+F'_n\qquad R_n(u)=r(u)+F''_n
$$ where $h(u),\, r(u)$ are specific constants for each urn $u$ and
$F'_n,\,F''_n$ are random factors that are common to all the urns.
\\

Suppose now that there exists $k\in{\mathbb N}\setminus\{0\}$ with
$N_n(u)\vee R_n(u)\leq k\leq a(u)+b(u)$ for each $n$ and $u$ and that
the additional assumptions stated in Section \ref{conf-int} are
satisfied for each $u$. Set $m(u)=E[R_n(u)]$, $q(u)=E[R_n(u)^2]$,
$\mu(u)=E[N_n(u)]$ and $\eta(u)=E[N_n(u)^{-1}]$.  Denoting by
$M_n=[M_n(u)]_{u\in{\mathcal U}}$ the vector containing the empirical
mean of $X'_j(u)$ up to time $n$ for each urn $u$ and by
$Z=[Z(u)]_{u\in{\mathcal U}}$ the vector containing the almost sure
limit of $Z_n(u)$ (and $M_n(u)$) for each $u$, we have as a
consequence of (\ref{cond-ind}) that, for any vector
$\alpha=[\alpha(u)]_{u\in{\mathcal U}}$ of real numbers, the sequence
$\sqrt{n}\langle \alpha, (Z_n-Z)\rangle$\footnote{ The symbol
  $\langle\cdot,\cdot\rangle$ denotes the scalar product between two
  vectors.} converges in the sense of the almost sure conditional
convergence with respect to $\mathcal F$ to ${\mathcal
  N}(0,\sum_{u\in{\mathcal U}} \alpha(u)^2 V(u))$, where
$V(u)=\frac{q(u)}{m(u)^{2}\mu(u)}Z(u)(1-Z(u))$ and $\sqrt{n}\langle
\alpha, (M_n-Z)\rangle$ converges stably to ${\mathcal N}(0,
\sum_{u\in{\mathcal U}}\alpha(u)^2(U(u)+V(u)))$, where
$U(u)=\left(\frac{q(u)}{m(u)^{2}\mu(u)}+\eta(u)-
\frac{2}{\mu(u)}\right)Z(u)(1-Z(u))$.  Similarly as done in the
previous section, these convergence results can be useful in order to
get asymptotic confidence intervals for the linear combination
$\langle \alpha,Z\rangle$ of the limit proportions $Z(u)$.  \\

\indent Finally, the above results can be employed in order to obtain
asymptotic critical regions for tests. For instance, in order to
perform a statistical test with
\begin{equation*}
H_0:\; m(u)\geq \hbox{card}({\mathcal U}')^{-1}
\sum_{v\in{\mathcal U}'}m(v)\quad\hbox{against}\quad
H_1:\; m(u)<\hbox{card}({\mathcal U}')^{-1}\sum_{v\in{\mathcal U}'}m(v) 
\end{equation*}
where ${\mathcal U}'\subset{\mathcal U}$ and $u\notin{\mathcal U}'$,
we can use the asymptotic critical region
$$
\left\{ 
\frac{ 
\sqrt{
\hbox{card}({\mathcal U}')^{-1}
\sum_{v\in{\mathcal U}'}m_n(v)
}
}
{ \sqrt{m_n(u)} }
\frac{ \sqrt{n} \, |M_{n}(u)-Z_{n}(u)|}{ \sqrt{U_{n}(u)} }
>q_{ 1- \frac{\alpha}{2} }
\right\}
$$
where
\begin{equation}
U_{n}(u)=
\left(
\frac{q_n(u)}{m_n(u)^2\mu_n(u)}+\eta_n(u)-\frac{2}{\mu_n(u)}
\right)
Z_n(u)(1-Z_n(u))
\end{equation}
with 
\begin{equation}
\begin{split}
m_n(u)&=\frac{\sum_{j=1}^n R_j(u)}{n},\qquad
q_n(u)=\frac{\sum_{j=1}^n R_j(u)^2}{n}
\\
\mu_n(u)&=\frac{\sum_{j=1}^n N_j(u)}{n},\qquad
\eta_n(u)=\frac{\sum_{j=1}^n N_j(u)^{-1}}{n}.
\end{split}
\end{equation}

\noindent {\bf Acknowledgments}
\\
\noindent Irene Crimaldi acknowledges support from CNR PNR Project
``CRISIS Lab''.  Moreover, she is a member of the Italian group ``Gruppo
Nazionale per l'Analisi Matematica, la Probabilit\`a e le loro
Applicazioni (GNAMPA)'' of the Italian In\-sti\-tu\-te ``Istituto Nazionale
di Alta Matematica (INdAM)''.

\appendix

\section{Some auxiliary results}

For reader's convenience, we state here some results used in the
proofs.

\begin{lem-app}\label{lemma-app-C} (Th. 2 in \cite{Bl-Du} or  
a special case of Lemma A.2 in \cite{C})\\
Let $\mathcal F$ be a filtration and set ${\mathcal
  F}_{\infty}=\bigvee_n {\mathcal F}_n$.  Then, for each sequence
$(Y_n)$ of integrable complex random variables, which is dominated in
$L^1$ and which converges almost surely to a complex random variable
$Y$, the conditional expectation $E[Y_n|{\mathcal F}_n]$ converges
almost surely to the conditional expectation $E[Y| {\mathcal
    F}_{\infty} ]$.
\end{lem-app}

\begin{lem-app}\label{lemma-app} (Lemma 2 in \cite{BCPR-urne1})\\
Let $(Y_n)$ be a sequence of real random variables, adapted to a
filtration $\mathcal F$. If $\sum_{j\geq 1} j^{-2} E[Y_j^2]<+\infty$
and $E[Y_{j}| {\mathcal F}_{j-1}]\stackrel{a.s.}\longrightarrow Y$ for
some random variable $Y$, then
$$
n\sum_{j\geq n}\frac{Y_j}{j^2}\stackrel{a.s.}\longrightarrow Y,
\qquad
\frac{1}{n}\sum_{j=1}^n Y_j\stackrel{a.s.}\longrightarrow Y.
$$
\end{lem-app}

\begin{thm-app}\label{th-main-app} 
(Special case of Th.~1 together with Prop.~1 in \cite{BCPR-urne1}
  and Th.~10 in \cite{BCPR-indian})\\ Let $(Y_n)$ be a bounded sequence
  of real random variables, adapted to a filtration ${\mathcal
    G}=(\mathcal{G}_n)$. Set
\begin{equation*}
M_n=\frac{1}{n}\sum_{j=1}^n Y_j\quad\text{and}\quad 
Z_n=E[Y_{n+1}|\mathcal{G}_n].
\end{equation*}
Suppose that $(Z_n)$ is a $\mathcal G$-martingale.\\

\indent Then, $Z_n\overset{a.s./L^1}\longrightarrow Z$ and
$M_n\overset{a.s./L^1}\longrightarrow Z$ for some real random variable
$Z$. Moreover, $\sqrt{n}(Z_n-Z)$ converges in the sense of the almost
sure conditional convergence with respect to $\mathcal G$ toward the
Gaussian kernel~${\mathcal N}(0,V)$ for some random variable $V$,
provided
\begin{itemize}
\item[c1)] $E\left[\sup_{j\geq
    1}\sqrt{j}\,|Z_{j-1}-Z_j|\,\right]<+\infty$, 
\item[c2)] $n\sum_{j\geq n}(Z_{j-1}-Z_j)^2\stackrel{a.s.}\longrightarrow
V$.
\end{itemize}
If condition
\begin{itemize}
\item[c3)] $n^{-1}\sum_{j=1}^n\bigl[Y_j-Z_{j-1}+j(Z_{j-1}-Z_j)\bigr]^2
  \stackrel{P}\longrightarrow U$
\end{itemize}
is also satisfied for some random variable $U$, then
\begin{equation*}
\left[\sqrt{n}\,\bigl(M_n-Z_n\bigr),\sqrt{n}(Z_n-Z)\right]
\stackrel{stably}\longrightarrow
\mathcal{N}\bigl(0,\,U\bigr)\otimes{\mathcal N}(0,V).
\end{equation*}
\end{thm-app}

\end{document}